\documentclass[12pt, reqno]{article}
\usepackage{amsmath}
\usepackage{amssymb}
\usepackage{latexsym,bm}
\usepackage{amsthm}
\usepackage{graphicx}
\usepackage{cite}
\usepackage[colorlinks,urlcolor=blue]{hyperref}

\usepackage[top=3cm,bottom=3cm,left=3cm,right=3cm,includehead,includefoot]{geometry}
\usepackage{algorithm}
\usepackage{algorithmic}
\allowdisplaybreaks

\newtheorem{thm}{Theorem}[section]
\newtheorem{cor}[thm]{Corollary}
\newtheorem{lem}{Lemma}[section]
\newtheorem{prop}[thm]{Proposition}
\newtheorem{defn}[thm]{Definition}
\theoremstyle{Condition}

\theoremstyle{remark}
\newtheorem{rmk}{Remark}[section]
\theoremstyle{problem}
\newtheorem{pro}{Problem}[section]
\theoremstyle{Notation}
\newtheorem{nota}{Notation}[section]
\numberwithin{equation}{section}
\theoremstyle{example}

\numberwithin{equation}{section}

\begin{document}

\bigskip
\bigskip

\bigskip

\begin{center}

\textbf{\large A randomized intertial primal-dual fixed point algorithm for monotone
inclusions}

\end{center}

\begin{center}
Meng Wen $^{1,2}$, Shigang Yue$^{4}$, Yuchao Tang$^{3}$, Jigen Peng$^{1,2}$
\end{center}

\begin{center}
1. School of Mathematics and Statistics, Xi'an Jiaotong University,
Xi'an 710049, P.R. China \\
2. Beijing Center for Mathematics and Information Interdisciplinary
Sciences, Beijing, P.R. China

 3. Department of Mathematics, NanChang University, Nanchang
330031, P.R. China\\
4. School of Computer Science, University of Lincoln, LN6 7TS, UK
\end{center}

\footnotetext{\hspace{-6mm}$^*$ Corresponding author.\\
E-mail address:jgpengxjtu@126.com}

\bigskip

\noindent  \textbf{Abstract} In this paper, we propose a randomized intertial block-coordinate primal-dual fixed point algorithm to solve a wide array of monotone inclusion problems base on the modification
of the heavy ball method of Nesterov [41]. These methods rely on a sweep of blocks of variables which are activated at each iteration according to a random rule. To this end we formulate the inertial version of the Krasnosel'skii-Mann algorithm for approximating the set of fixed points of a quasinonexpansive operator, for which we also provide an exhaustive convergence analysis. As a by-product, we can obtain some intertial block-coordinate operator splitting methods for solving composite monotone inclusion and convex minimization problems.


\noindent \textbf{MR(2000) Subject Classification} 47H09, 90C25,

\section{Introduction}

The problem of approaching the set of zeros of a sum of
monotone operators or minimizing a sum of proper lower-semicontinuous convex functions by means of primal-dual algorithms, where various linear operators are involved in the formulation, solving jointly its primal and dual forms and none of the linear operators needs to be inverted, continues to be a very attractive research area. This is due to its avoiding such inversions may
offer a significant advantage in terms of computational complexity when dealing with large-scale
problems [20-34] and its applicability in
the context of solving real-life problems which can be modeled as nondifferentiable convex optimization problems, like
those arising in image processing, signal recovery, support vector machines classification, location theory, clustering,
network communications, etc.

The main advantage of inertial methods is it can  change greatly
improves the performance of the scheme. These benefits have long been recognized [18, 35-37] and have become increasingly important in many monotone inclusion problems. In instances, Radu Ioan, Ern\"{o} Robert and Christopher [1] considered an inertial Douglas-Rachford splitting for finding the set of zeros of
the sum of two maximally monotone operators in Hilbert spaces and investigate its convergence.

Recently, Patrick L. and Jean-Christophe [2] introduced a block-coordinate fixed point algorithms with applications to nonlinear analysis and optimization in Hilbert spaces based on a notion of stochastic
quasi-Fej\'{e}r monotonicity. The algorithms were composed of quasinonexpansive operators or compositions of averaged nonexpansive
operators. In addition, they proved weak and strong convergence results of
the sequences generated by these algorithms.

Motivated and inspired by the above results, we introduce a  randomized intertial block-coordinate primal-dual fixed point algorithm to solve  monotone inclusion problems. The main benefit of block-coordinate algorithms is to result in implementations with reduced
complexity and memory requirements per iteration (see [38-40]). We obtain the weak and
strong convergence theorem of proposed algorithms. Furthermore, we using our algorithms to solve some monotone inclusion problems and large-scale convex minimization problems. Our results
also extend and improve the corresponding results of Patrick L. and Jean-Christophe [2] and Jean-Christophe  and Audrey  [3] and many
others.

The rest of this paper is organized as follows. In the next section,
 we recall the conception of  stochastic quasi-Fej\'{e}r monotone and some related notations and then deduce the  idea we
proposed stochastic inertial quasi-Fej\'{e}r monotonicity. Furthermore, we prove almost sure convergence results for  a  randomized intertial iteration scheme. In section 3,
we using this scheme to design a stochastic intertial block-coordinate fixed point
algorithms for relaxed iterations of quasinonexpansive operators and a stochastic intertial block-coordinate algorithms involving compositions of averaged nonexpansive operators. In section 4, Based on the algorithm in Section 3,
we present a preconditioned stochastic intertial block-coordinate
forward-backward algorithm.
In the final section, we consider the application of the presented
algorithms. First, we give a novel intertial block-coordinate prima-dual
algorithms for solving a zero of a sum of monotone operators. Moreover, we prove
the convergence of proposed algorithms. Second, an intertial block-coordinate primal-dual algorithms are designed
to solve composite convex optimization problems, and we study
their convergence.

\section{Stochastic inertial quasi-Fej\'{e}r monotonicity}

In various areas of nonlinear analysis and optimization to
unify the convergence proofs of deterministic algorithms, Fej\'{e}r monotonicity has been exploited far-ranging see, e.g., [21-24]. In the late
1960s, this conception was reconsidered in a stochastic setting in Euclidean spaces [25-27]. In this
section, we introduce a conception of stochastic inertial quasi-Fej\'{e}r monotone sequence in Hilbert spaces and
use the results to a general stochastic inertial iterative method.  We will use the following
notation throughout the paper.

\begin{nota}
 Let $H$ be a separable real Hilbert space with inner product $\langle\cdot,\cdot\rangle$, norm $\|\cdot\|$ and Borel $\sigma$-algebra $\mathcal{B}$. In $H$, we write $\rightarrow$ and $\rightharpoonup$
 indicate respectively weak and strong convergence. The sets of strong and weak cluster point of a sequence
$(x_{n})_{n\in\mathbb{N}}$ in  $H$ are  recorded as $\mathfrak{Q}(x_{n})_{n\in\mathbb{N}}$ and $\mathfrak{Y}(x_{n})_{n\in\mathbb{N}}$ respectively. $(\Omega,\mathcal{F},P)$ is the
underlying probability space. A $H$-valued random variable is a measurable
map $x:(\Omega,\mathcal{F})\rightarrow(H,\mathcal{B})$. The smallest $\sigma$-algebra generated by a family $\Phi$ of random variables is
denoted by $\sigma(\Phi)$. The expectation is denoted by $E(\cdot)$. Let $\mathfrak{F} = (\mathcal{F})_{n\in\mathbb{N}}$ be a sequence of
sub-sigma algebras of $\mathcal{F}$ such that $(\forall n\in\mathbb{N})\mathcal{F}_{n}\subset\mathcal{F}_{n+1}$. We denote by $\ell_{+}(\mathfrak{F})$ the set of sequences
of $[0,+\infty)$-valued random variables $(\pi_{n})_{n\in\mathbb{N}}$ such that, for every $n\in\mathbb{N}$, $\pi_{n}$ is $\mathcal{F}_{n}$-measurable. We set
$$(\forall p\in (0,+\infty)) ~~ \ell_{+}^{p}(\mathfrak{F})=\{(\pi_{n})_{n\in\mathbb{N}}\in\ell_{+}(\mathfrak{F})|\sum_{n\in\mathbb{N}}\pi_{n}^{p}<+\infty  ~~P-a.s.\},$$
and
$$\ell_{+}^{\infty}(\mathfrak{F})=\{(\pi_{n})_{n\in\mathbb{N}}\in\ell_{+}(\mathfrak{F})|\sup_{n\in\mathbb{N}}\pi_{n}<+\infty ~~ P-a.s.\}.$$
\end{nota}
Let $(x_{n})_{n\in\mathbb{N}}$ be sequences of $H$-valued random variables, we write for\\
$\mathfrak{X}=(\mathcal{X}_{n})_{n\in\mathbb{N}}$, where $(\forall n\in\mathbb{N})$   $\mathcal{X}_{n}=\sigma(x_{0},\cdots,x_{n})$.
\begin{lem}
([21, Corollary 2.14]). For $\forall a\in \mathbb{R}$ and  $\forall x,y\in H$,
$$\|ax+(1-a)y\|^{2}=a\|x\|^{2}+(1-a)\|y\|^{2}-a(1-a)\|x-y\|^{2}.\eqno{(2.1)}$$
\end{lem}
\begin{lem}
([2]). Let $D$ be a nonempty closed subset of $H$, let $\Phi: [0,+\infty) \rightarrow [0,+\infty)$ be a strictly
increasing function such that $\lim_{t\rightarrow+\infty}\Phi(t)=+\infty$, and let $(x_{n})_{n\in\mathbb{N}}$ be a sequence of $H$-valued random
variables. Suppose that, for every $z\in D$, there exist $(\zeta_{n}(z))_{n\in\mathbb{N}}\in\ell_{+}^{1}(\mathfrak{X})$, $(\xi_{n}(z))_{n\in\mathbb{N}}\in\ell_{+}(\mathfrak{X})$, and $(\theta_{n}(z))_{n\in\mathbb{N}}\in\ell_{+}^{1}(\mathfrak{X})$ such that the following is satisfied P-a.s.:
$$E(\Psi(\|x_{n+1}-z\|)|\mathcal{X}_{n})+\zeta_{n}(z)\leq(1+\xi_{n}(z))\Psi(\|x_{n}-z\|)+\theta_{n}(z).\eqno{(2.2)}$$
Then the following hold:\\
(i) $\sum_{n\in\mathbb{N}}\xi_{n}(z)<+\infty$ P-a.s., $\forall z\in D$.\\
(ii) $(x_{n})_{n\in\mathbb{N}}$ is bounded P-a.s.\\
(iii) There exists $\bar{\Omega}\in\mathcal{F}$ such that $P(\bar{\Omega}) = 1$ and, for every $\omega \in \bar{\Omega}$
and every $z \in D$, $(\|x_{n}(\omega)-z\|)_{n\in\mathbb{N}}$ converges.\\
(iv) Suppose that $\mathfrak{Q}(x_{n})_{n\in\mathbb{N}}\subset D$  P-a.s. Then $(x_{n})_{n\in\mathbb{N}}$ converges weakly P-a.s. to a $D$-valued random
variable.\\
(v) Suppose that $\mathfrak{Y}(x_{n})_{n\in\mathbb{N}}\cap D\neq\emptyset$ P-a.s. Then  $(x_{n})_{n\in\mathbb{N}}$ converges strongly P-a.s. to a $D$-valued
random variable.\\
(vi) Suppose that  $\mathfrak{Y}(x_{n})_{n\in\mathbb{N}}\neq\emptyset$ P-a.s. and that $\mathfrak{Q}(x_{n})_{n\in\mathbb{N}}\subset D$ P-a.s. Then $(x_{n})_{n\in\mathbb{N}}$ converges
strongly P-a.s. to a $D$-valued random variable.

\end{lem}
\begin{lem}
(Opial). Let $C$ be a nonempty set of $H$ and $(x_{n})_{n\in\mathbb{N}}$ be a sequence in $H$ such that the following two conditions hold:\\
(a) for every $x\in C$; $\lim_{n\rightarrow+\infty}\|x_{n}-x\|$ exists.\\
(b) every sequential weak cluster point of $(x_{n})_{n\in\mathbb{N}}$ is in $C$.\\
Then $(x_{n})_{n\in\mathbb{N}}$ converges weakly to a point in $C$.

\end{lem}
\begin{lem}
([20, Theorem 3]). We denote by $\mathcal{A}(H,\beta)$ the set of $\beta$-averaged operators on $H$. Let $\beta_{1} \in (0, 1)$, $\beta_{2} \in (0, 1)$, $T_{1}\in\mathcal{A}(H,\beta_{1})$, and $T_{2}\in\mathcal{A}(H,\beta_{2})$. Then $T_{1}\circ T_{2}\in\mathcal{A}(H,\beta')$, where
$$\beta'=\frac{\beta_{1}+\beta_{2}-2\beta_{1}\beta_{2}}{1-\beta_{1}\beta_{2}}.$$

\end{lem}

\begin{thm}
Let $C$ be a nonempty closed affine subset of $H$, let $(\lambda_{n})_{n\in\mathbb{N}}$ be a sequence in $(0, 1]$, and let
$(t_{n})_{n\in\mathbb{N}}$ and $(x_{n})_{n\in\mathbb{N}}$ be sequences of $C$-valued random variables. Suppose that the following
hold:\\
(i) $w_{n}=x_{n}+\alpha_{n}(x_{n}-x_{n-1})$, $x_{n+1}=w_{n}+\lambda_{n}(t_{n}-w_{n})$, $\forall n\in\mathbb{N}$.\\
(ii) $x_{0}$, $x_{1}$ are arbitrarily chosen in $C$, $(\alpha_{n})_{n\geq1}$ is nondecreasing with $\alpha_{1}=0$ and $0\leq\alpha_{n}\leq\alpha<1$, $\forall n\geq1$ and $\lambda, \tau, \delta>0$ are such that\\
$\delta>\frac{\alpha^{2}(1+\alpha)+\alpha\tau}{1-\alpha^{2}}$ and $0<\lambda\leq\lambda_{n}\leq\frac{\delta-\alpha[\alpha(1+\alpha)+\alpha\delta+\tau]}{\delta[1+\alpha(1+\alpha)+\alpha\delta+\tau]}$, $\forall n\geq1$.\\
(iii) $\forall n\in\mathbb{N}$, $E(\|t_{n}-y\|^{2}|\mathcal{X}_{n})\leq\|w_{n}-y\|^{2}$ is satisfied
P-a.s.\\
Then
 \par
 $\sum_{n\in\mathbb{N}}E(\|t_{n}-w_{n}\|^{2}|\mathcal{X}_{n})<+\infty$  P-a.s.\\
Moreover, assuming that:\\
(iv) $\mathfrak{Q}(x_{n})_{n\in\mathbb{N}}\subset C$ P-a.s.
Then $(x_{n})_{n\in\mathbb{N}}$  converges weakly P-a.s. to a $C$-valued random variable $\hat{x}$. Furthermore, if\\
(v) $\mathfrak{Y}(x_{n})_{n\in\mathbb{N}}\neq\emptyset$ P-a.s.,\\
then $(x_{n})_{n\in\mathbb{N}}$ converges strongly P-a.s. to $\hat{x}$.
\end{thm}

\begin{proof}
We proceed with the following steps.\\
Step 1. Because of the choice of $\delta$, $\lambda_{n}\in(0,1) \forall n\geq1$. Moreover, since $C$ is affine, we can know that the iterative scheme provides a well-defined sequence in $C$. Let $y\in C$ and $n\geq1$.
From (2.1) and (iii), we can know that
\begin{align*}
\|x_{n+1}-y\|^{2}&=\|(1-\lambda_{n})w_{n}+\lambda_{n}t_{n}-y\|^{2}\\
&=\|(1-\lambda_{n})(w_{n}-y)+\lambda_{n}(t_{n}-y)\|^{2}\\
&=(1-\lambda_{n})\|w_{n}-y\|^{2}+\lambda_{n}\|t_{n}-y\|^{2}-\lambda_{n}(1-\lambda_{n})\|t_{n}-w_{n}\|^{2}\\
&\leq\|w_{n}-y\|^{2}-\lambda_{n}(1-\lambda_{n})\|t_{n}-w_{n}\|^{2}.\tag{2.3}
\end{align*}
Using (2.1) again, we obtain
\begin{align*}
\|w_{n}-y\|^{2}&=\|(1+\alpha_{n})(x_{n}-y)-\alpha_{n}(x_{n-1}-y)\|^{2}\\
&=(1+\alpha_{n})\|x_{n}-y\|^{2}-\alpha_{n}\|x_{n-1}-y\|^{2}+\alpha_{n}(1+\alpha_{n})\|x_{n}-x_{n-1}\|^{2}.\tag{2.4}
\end{align*}
So from (2.3), we have
\begin{align*}
(\forall n\geq1)&\mathrm{E}(\|x_{n+1}-y\|^{2}|\mathcal{X}_{n})-(1+\alpha_{n})\|x_{n}-y\|^{2}+\alpha_{n}\|x_{n-1}-y\|^{2}\\
&\leq-\lambda_{n}(1-\lambda_{n})\mathrm{E}(\|t_{n}-w_{n}\|^{2}|\mathcal{X}_{n})+\alpha_{n}(1+\alpha_{n})\|x_{n}-x_{n-1}\|^{2}.\tag{2.5}
\end{align*}
Furthermore, we have
\begin{align*}
\mathrm{E}(\|t_{n}-w_{n}\|^{2}|\mathcal{X}_{n})&=\|\frac{1}{\lambda_{n}}(x_{n+1}-x_{n})+\frac{\alpha_{n}}{\lambda_{n}}(x_{n-1}-x_{n})\|^{2}\\
&=\frac{1}{\lambda_{n}^{2}}\|x_{n+1}-x_{n}\|^{2}+\frac{\alpha_{n}^{2}}{\lambda_{n}^{2}}\|x_{n-1}-x_{n}\|^{2}+2\frac{\alpha_{n}}{\lambda_{n}^{2}}\langle x_{n+1}-x_{n},x_{n-1}-x_{n}\rangle\\
&\geq\frac{1}{\lambda_{n}^{2}}\|x_{n+1}-x_{n}\|^{2}+\frac{\alpha_{n}^{2}}{\lambda_{n}^{2}}\|x_{n-1}-x_{n}\|^{2}\\
&+\frac{\alpha_{n}}{\lambda_{n}^{2}}(-\upsilon_{n}\|x_{n+1}-x_{n}\|^{2}-\frac{1}{\upsilon_{n}}\|x_{n-1}-x_{n}\|^{2})\\
&=\frac{1-\alpha_{n}\upsilon_{n}}{\lambda_{n}^{2}}\|x_{n+1}-x_{n}\|^{2}+\frac{\alpha_{n}^{2}\upsilon_{n}-\alpha_{n}}{\lambda_{n}^{2}\upsilon_{n}}\|x_{n-1}-x_{n}\|^{2},\tag{2.6}
\end{align*}
where $\upsilon_{n}=\frac{1}{\alpha_{n}+\delta\lambda_{n}}$.\\
Put (2.6) into (2.5), we have
\begin{align*}
&\mathrm{E}(\|x_{n+1}-y\|^{2}|\mathcal{X}_{n})-(1+\alpha_{n})\|x_{n}-y\|^{2}+\alpha_{n}\|x_{n-1}-y\|^{2}\\
&\leq\frac{(1-\lambda_{n})(\alpha_{n}\upsilon_{n}-1)}{\lambda_{n}}\|x_{n+1}-x_{n}\|^{2}+\mu_{n}\|x_{n}-x_{n-1}\|^{2},\tag{2.7}
\end{align*}
where
$$\mu_{n}=\alpha_{n}(1+\alpha_{n})+\alpha_{n}(1-\lambda_{n})\frac{1-\upsilon_{n}\alpha_{n}}{\upsilon_{n}\lambda_{n}}>0,\eqno{(2.8)}$$
since $\upsilon_{n}\alpha_{n} < 1$ and $\lambda_{n}\in(0,1)$.\\
Now, we consider the choice of $\upsilon_{n}$ and set
$$\delta=\frac{1-\upsilon_{n}\alpha_{n}}{\upsilon_{n}\lambda_{n}}.$$
Therefore, from (2.8) we can know taht
$$\mu_{n}=\alpha_{n}(1+\alpha_{n})+\alpha_{n}(1-\lambda_{n})\delta\leq\alpha(1+\alpha)+\alpha\delta ~~~~~\forall n\geq1. \eqno{(2.9)}$$
In the following, we set $\psi_{n}=\|x_{n}-y\|^{2}$, $\forall n\in \mathbb{N}$ and $\eta_{n+1}=\mathrm{E}(\psi_{n+1}|\mathcal{X}_{n})-\alpha_{n+1}\psi_{n}+\mu_{n+1}\|x_{n+1}-x_{n}\|^{2}$, $\forall n\geq1$.
With the same proof of [1], we can obtain that
$$\sum_{n\in\mathbb{N}}\|x_{n+1}-x_{n}\|^{2}<+\infty.\eqno{(2.10)}$$
We have proven above that for an arbitrary $y\in C$ the
inequality (2.7) is true. By Step 1, (2.9) and Lemma 2.2 we derive that exists $\Omega \in C$ such that $P(\Omega) = 1$ and, for every $\omega \in \Omega$
and every $y \in C$ $\lim_{n\rightarrow\infty}\|x_{n}(\omega)-y\|$ exists (in (2.7), we still consider that $\upsilon_{n}\alpha_{n}<1$, $\forall n\geq1$).
On the other hand, from (i) we have
\begin{align*}
(\forall n\geq1)\mathrm{E}(\|t_{n}-w_{n}\||\mathcal{X}_{n})&=\frac{1}{\lambda_{n}}\mathrm{E}(\|x_{n+1}-w_{n}\||\mathcal{X}_{n})\\
&\leq\frac{1}{\lambda}\mathrm{E}(\|x_{n+1}-w_{n}\||\mathcal{X}_{n})\\
&\leq\frac{1}{\lambda}(\|x_{n+1}-x_{n}\|+\alpha\|x_{n}-x_{n-1}\|),\tag{2.11}
\end{align*}
 therefore, by (2.10) we have \\
 $\sum_{n\in\mathbb{N}}E(\|t_{n}-w_{n}\|^{2}|\mathcal{X}_{n})<+\infty$  P-a.s.\\
Step 2. We will prove the  convergence of $(x_{n})_{n\in\mathbb{N}}$. \\
By the definition of $w_{n}$, we can know that
\begin{align*}
\|w_{n}-y\|^{2}&=\|(x_{n}-y)-\alpha_{n}(x_{n}-x_{n-1})\|^{2}\\
&\leq\|x_{n}-y\|^{2}+\alpha_{n}^{2}\|x_{n}-x_{n-1}\|^{2}+2\alpha_{n}\|x_{n}-y\|\|x_{n}-x_{n-1}\|.\tag{2.12}
\end{align*}
Put (2.12) into (2.3), we have
\begin{align*}
\|x_{n+1}-y\|^{2}&\leq\|x_{n}-y\|^{2}+\alpha_{n}^{2}\|x_{n}-x_{n-1}\|^{2}+2\alpha_{n}\|x_{n}-y\|\|x_{n}-x_{n-1}\|\\
&-\lambda_{n}(1-\lambda_{n})\|t_{n}-w_{n}\|^{2}.\tag{2.13}
\end{align*}
Set
 $$\xi_{n}(y)=\lambda_{n}(1-\lambda_{n})\|t_{n}-w_{n}\|^{2}$$
and
$$\theta_{n}(y)=\alpha_{n}^{2}\|x_{n}-x_{n-1}\|^{2}+2\alpha_{n}\|x_{n}-y\|\|x_{n}-x_{n-1}\|.$$
So, from (iv) and Lemma 2.2 (iv) applied with $\Psi :t\rightarrow t^{2}$, we obtain that $(x_{n})_{n\in\mathbb{N}}$  converges weakly P-a.s. to a $C$-valued random variable $\hat{x}$.
Likewise, from (iv)-(v) and Lemma 2.2 (vi) applied with $\Psi :t\rightarrow t^{2}$, we obtain the strong
convergence of $(x_{n})_{n\in\mathbb{N}}$.
\end{proof}
\begin{defn}
 An operator $T: H \rightarrow H$ is nonexpansive, if the following inequality holds for any $x,y\in H$:
$$\|Tx-Ty\|\leq\|x-y\|.$$
A point $x\in H$ is a fixed point of $T$ provided $Tx = x$. Denote by $Fix(T)$ the set of fixed points of
$T$, that is, $Fix(T)=\{x\in H, Tx = x\}$. Let $(x_{n})_{n\in\mathbb{N}}$ be a sequence in $ H$ and $z\in H$ such that $x_{n}\rightharpoonup z$ and $Tx_{n}-x_{n}\rightharpoonup 0$ as $n\rightarrow+\infty$. Then $z\in Fix(T)$. This is called $T$ is demicompact at  $z\in H$ [21, Corollary 4.18].
\end{defn}
By Theorem 2.1, we can obtain the following Corollary.
\begin{cor}
Let $T: H \rightarrow H$ be a nonexpansive operator such that $Fix(T)\neq\emptyset$, and $(\lambda_{n})_{n\in\mathbb{N}}$ be a sequence in $[0, 1]$.
Let $x_{0}, x_{1}$ be $H$-valued random variables which are  chosen arbitrarily. For $n\geq0$:
$$
 \left\{
\begin{array}{l}
w_{n}=x_{n}+\alpha_{n}(x_{n}-x_{n-1}),\\
x_{n+1}=x_{n}+\lambda_{n}(Tx_{n}-x_{n}).
\end{array}
\right.
$$
\end{cor}
Then the following hold:\\
(i) $(x_{n})_{n\in\mathbb{N}}$ converges weakly P-a.s. to a $Fix (T)$-valued random variable.\\
(ii) Suppose that $T$ is demicompact at 0 (see Definition 2.1). Then  $(x_{n})_{n\in\mathbb{N}}$ converges strongly P-a.s.
to a $Fix( T)$-valued random variable.

\begin{proof}
Set $D=Fix(T)$. Since $T$ is continuous, $T$ is measurable and $D$ is closed. Now let $y\in D$ and set
$(n\in\mathbb{N})$, $t_{n} = Tw_{n}$. Then, by the nonexpansiveness of $T$, we have for every $n\in\mathbb{N}$

$$
 \left\{
\begin{array}{l}
x_{n+1}=w_{n}+\lambda_{n}(t_{n}-w_{n}),\\
\sum_{n\in\mathbb{N}}E(\|t_{n}-w_{n}\|^{2}|\mathcal{X}_{n})=\|Tw_{n}-w_{n}\|^{2},\\
\sum_{n\in\mathbb{N}}E(\|t_{n}-y\|^{2}|\mathcal{X}_{n})=\|Tw_{n}-Ty\|^{2}\leq\|w_{n}-y\|^{2}.
\end{array}
\right.\eqno{(2.14)}
$$
It is satisfied  properties (i)-(iii) of Theorem 2.1.  Therefore, from (2.14) and the conclusion of (iii) in Theorem 2.1 , we know that exists $\bar{\Omega}\in\mathcal{F}$ such that $P(\bar{\Omega}) = 1$ and
 $$\sum_{n\in\mathbb{N}}\|Tw_{n}(\omega)-w_{n}(\omega)\|^{2}<+\infty, \forall\omega\in\bar{\Omega}.\eqno{(2.15)}$$
With the same proof of Step1 in Theorem 2.1, we can know that for
every $y \in C$, $\forall\omega\in\bar{\Omega}$, $\lim_{n\rightarrow\infty}\|x_{n}(\omega)-y\|$ exists.
On the other hand, let $x$ be a sequential weak cluster point of $(x_{n})_{n\in\mathbb{N}}$, that is, the latter has
a subsequence $(x_{n_{k}})_{k\in\mathbb{N}}$ such that $x_{n_{k}}\rightharpoonup x$ as $k\rightarrow\infty$. By (2.10) and the definition of $w_{n}$ and the upper bound for $\alpha_{n}$, we get
$w_{n_{k}}\rightharpoonup x$.
Then, by (2.15) we know there exists $\hat{\Omega}\subset\bar{\Omega}$ such that $\hat{\Omega}\in\mathcal{F}$ such that $P(\hat{\Omega}) = 1$ and, for every $\omega\in\hat{\Omega}$
$Tw_{n_{k}}(\omega)-w_{n_{k}}(\omega)\rightharpoonup 0$ as $k\rightarrow\infty$. Applying now Definition 2.1 for the sequence $(w_{n_{k}}(\omega))_{k\in\mathbb{N}}$ we conclude that
$x\in Fix(T)$. Since the two assumptions of Lemma 2.3 are verified, it follows that $(x_{n}(\omega))_{n\in\mathbb{N}}$ converges weakly to a point in $Fix(T)$.

\end{proof}

\section{A stochastic intertial block-coordinate fixed point algorithms}

\begin{nota}
Let $H_{1}, \cdots,H_{m}$ are separable real Hilbert spaces and $\mathbf{H }= H_{1} \oplus \cdots \oplus H_{m}$ be their direct
Hilbert sum. The inner products and  norms of these spaces are all denoted by $\langle\cdot,\cdot\rangle$ and $\|\cdot\|$, and $\mathbf{x}=(x_{1}, \cdots,x_{m})$ indicates a generic vector in $\mathbf{H }$.
Let $(x_{n})_{n\in\mathbb{N}}$ be a sequence of $H$-valued random variables, we set $\mathcal{X}_{n}=\sigma(x_{0}, \cdots,x_{n})$.

\end{nota}

\begin{defn}
([21]) An operator $\mathbf{T}: \mathbf{H}\rightarrow \mathbf{H}$ is quasinonexpansive if
$$\|\mathbf{T}\mathbf{x}-\mathbf{z}\|\leq\|\mathbf{x}\mathbf{-z}\|, \forall \mathbf{z}\in Fix(\mathbf{T}), \forall \mathbf{x}\in \mathbf{H}.\eqno{(3.1)}$$
\end{defn}

\begin{thm}
Let $(\forall n\in\mathbb{N})$ $\mathbf{T}_{n}: \mathbf{H}\rightarrow \mathbf{H}: \mathbf{x}\mapsto (T_{i,n}\mathbf{x})_{1\leq i\leq m}$  be a quasinonexpansive
operator where, for every $i\in\{1,\cdots,m\}$, $T_{i,n} :\mathbf{H }\rightarrow H_{i}$ is measurable. Suppose that $(\lambda_{n})_{n\in\mathbb{N}}$ is a sequence in $(0, 1)$, and set $M=\{0,1\}^{m}\backslash \{\mathbf{0}\}$. Let $\mathbf{x}_{0}, \mathbf{x}_{1}$ be  $\mathbf{H}$-valued
random variables which are arbitrarily chosen, and let $(\varepsilon_{n})_{n\in\mathbb{N}}$ be identically distributed $M$-valued random variables. For $n\geq0$,
$$
 \left\{
\begin{array}{l}
\mathbf{w}_{n}=\mathbf{x}_{n}+\alpha_{n}(\mathbf{x}_{n}-\mathbf{x}_{n-1}),\\
x_{i,n+1}=w_{i,n}+\varepsilon_{i,n}\lambda_{n}(T_{i,n}(w_{1,n},\cdots,w_{m,n})-w_{i,n}), i=\{1, \cdots,m\},
\end{array}
\right.\eqno{(3.2)}
$$
and set $\mathcal{E}_{n}=\sigma(\varepsilon_{n})$. Moreover, suppose that the following hold:\\
(i) $\mathbf{D}=\bigcap_{n\in\mathbb{N}}Fix(\mathbf{T}_{n})\neq\emptyset$.\\
(ii) For every $n\in\mathbb{N}$, $\mathcal{E}_{n}$ and $\mathcal{X}_{n}$ are independent.\\
(iii) $p_{i}=P[\varepsilon_{i,0}=1]>0$, $\forall i\in\{1, \cdots,m\}$.\\
(iv) $(\alpha_{n})_{n\geq1}$ is nondecreasing with $\alpha_{1}=0$ and $0\leq\alpha_{n}\leq\alpha<1$, $\forall n\geq1$ and $\lambda, \tau, \delta>0$ are such that
$\delta>\frac{\alpha^{2}(1+\alpha)+\alpha\tau}{1-\alpha^{2}}$ and $0<\lambda\leq\lambda_{n}\leq\frac{\delta-\alpha[\alpha(1+\alpha)+\alpha\delta+\tau]}{\delta[1+\alpha(1+\alpha)+\alpha\delta+\tau]}$, $\forall n\geq1$.\\
Then
\par
$\mathbf{T}_{n}\mathbf{x}_{n}-\mathbf{x}_{n}\rightarrow0$ P-a.s.\\
In addition, assume that:\\
(v) $\mathfrak{Q}(\mathbf{x}_{n})_{n\in\mathbb{N}}\subset \mathbf{D}$ P-a.s.
Then $(\mathbf{x}_{n})_{n\in\mathbb{N}}$  converges weakly P-a.s. to a $\mathbf{D}$-valued random variable $\mathbf{\hat{x}}$. Furthermore, if\\
(vi) $\mathfrak{Y}(\mathbf{x}_{n})_{n\in\mathbb{N}}\neq\emptyset$ P-a.s.,\\
then $(\mathbf{x}_{n})_{n\in\mathbb{N}}$ converges strongly P-a.s. to $\mathbf{\hat{x}}$.
\end{thm}

\begin{proof}
We define the norm $\||\cdot|\|$ on $\mathbf{H}$ by
$$
\||\mathbf{x}|\|^{2}=\sum_{i=1}^{m}\frac{1}{p_{i}}\|x_{i}\|^{2},\,\,\, \forall \mathbf{x}\in\mathbf{H}.\eqno{(3.3)}
$$
We will use Theorem 2.1 in $(\mathbf{H}, \||\cdot|\|)$. For every $n\in\mathbb{N}$, set  $\mathbf{t}_{n}=(t_{i,n})_{1\leq i\leq m}$ where $t_{i,n}=w_{i,n}+\varepsilon_{i,n}(T_{i,n}\mathbf{w}_{n}-w_{i,n})$,  $\forall n\in \mathbb{N}$. Then from (3.2), we can know that
$$
(\forall n\in \mathbb{N}) \left\{
\begin{array}{l}
\mathbf{w}_{n}=\mathbf{x}_{n}+\alpha_{n}(\mathbf{x}_{n}-\mathbf{x}_{n-1}),\\
\mathbf{x}_{n+1}=\mathbf{w}_{n}+\lambda_{n}(\mathbf{t}_{n}-\mathbf{w}_{n}).
\end{array}
\right.\eqno{(3.4)}
$$
Since the operators $(\mathbf{T}_{n})_{ n\in \mathbb{N}}$ are quasinonexpansive, $\mathbf{D}$ is closed [20, Section 2]. Now let $\mathbf{y}\in\mathbf{D}$ and
set
$$
  q_{i,n}=:\mathbf{H}\times M\rightarrow \mathbb{R}:(\mathbf{w},\mathbf{s})\mapsto\|\mathrm{w}_{i}-\mathrm{y}_{i}+\epsilon_{i}(T_{i,n}\mathbf{w}-\mathrm{w}_{i})\|^{2},  \forall n\in \mathbb{N}, \forall i\in\{1, \cdots,m\}.\eqno{(3.5)}
$$
Note that, for every $n\in\mathbb{N}$ and every $i\in\{1, \cdots,m\}$, since $T_{i,n}$ is measurable, so are the functions $(q_{i,n}(\cdot, \mathbf{s}))_{\mathbf{s}\in M}$. With the same proof of Theorem 3.2 in [2], we can obtain that
\begin{align*}
(\forall n\in\mathbb{N})\mathrm{E}(\||\mathbf{t}_{n}-\mathbf{y}\||^{2}|\mathcal{X}_{n})&=\sum_{i=1}^{m}\frac{1}{p_{i}}\mathrm{E}(\|t_{i,n}-y_{i}\|^{2}|\mathcal{X}_{n})\\
&=\sum_{i=1}^{m}\frac{1}{p_{i}}\sum_{\mathbf{s}\in M}P[\varepsilon_{n}=\mathbf{s}]\|w_{i,n}-y_{i}+\varepsilon_{i}(T_{i,n}\mathbf{w}_{n}-w_{i,n})\|^{2}\\
&\leq\||\mathbf{w}_{n}-\mathbf{y}\||^{2}.\tag{3.6}
\end{align*}
It is obvious that for every $n\in\mathbb{N}$, conditions (i)-(iii) of Theorem 2.1 are satisfied. Therefore, we can
 derive from conclusion of (iii) in Theorem 2.1 that $\sum_{n\in\mathbb{N}}\mathrm{E}(\||\mathbf{t}_{n}-\mathbf{y}\||^{2}|\mathcal{X}_{n})<+\infty$ P-a.s.  This yields

$$\mathrm{E}(\||\mathbf{t}_{n}-\mathbf{y}\||^{2}|\mathcal{X}_{n}) \rightarrow0\eqno{(3.7)}$$
 P-a.s.\\
On the other hand, we can know that
\begin{align*}
(\forall n\in\mathbb{N})\mathrm{E}(\||\mathbf{t}_{n}-\mathbf{w}_{n}\||^{2}|\mathcal{X}_{n})&=\sum_{i=1}^{m}\frac{1}{p_{i}}\mathrm{E}(\|t_{i,n}-w_{i,n}\|^{2}|\mathcal{X}_{n})\\
&=\||\mathbf{T}_{n}\mathbf{w}_{n}-\mathbf{w}_{n}\||^{2}.\tag{3.8}
\end{align*}
From (3.7), we have $\mathbf{T}_{n}\mathbf{w}_{n}-\mathbf{w}_{n}\rightarrow0$  P-a.s. In addition, by the consequences of Theorem 2.1, we can obtain the weak and strong convergence.

\end{proof}

\begin{defn}
Let $T: H\rightarrow H$ be nonexpansive and let $\beta \in (0, 1)$. Then $T$ is averaged with constant
$\beta$, or $\beta$-averaged, if there exists a nonexpansive operator $R: H \rightarrow H$ such that $T = (1 -\beta)Id + \beta R$.
\end{defn}

\begin{prop}
([21, Proposition 4.25]). Let $T: H\rightarrow H$ be nonexpansive and let $\beta \in (0, 1)$. Then $T$ is  $\beta$-averaged, if and only if
$$\|Tx-Ty\|^{2}\leq\|x-y\|^{2}-\frac{1-\beta}{\beta}\|(I-T)x-(I-T)y\|^{2},~~\forall x,y\in H.$$
\end{prop}

\begin{thm}
Let $(\forall n\in\mathbb{N})$ $\mathbf{T}_{n}: \mathbf{H}\rightarrow \mathbf{H}: \mathbf{x}\mapsto (T_{i,n}\mathbf{x})_{1\leq i\leq m}$  be a $\beta_{n}$-averaged
operator where, for every $i\in\{1,\cdots,m\}$, $T_{i,n} :\mathbf{H }\rightarrow H_{i}$ and $\beta_{n}$ is a sequence in $(0, 1)$. Suppose that $b\in(0, 1)$,  and set $M=\{0,1\}^{m}\backslash \{\mathbf{0}\}$. Let $\mathbf{x}_{0}, \mathbf{x}_{1}$ be  $\mathbf{H}$-valued
random variables which are arbitrarily chosen, and let $(\varepsilon_{n})_{n\in\mathbb{N}}$ be identically distributed $M$-valued random variables. For $n\geq0$,
$$
 \left\{
\begin{array}{l}
\mathbf{w}_{n}=\mathbf{x}_{n}+\alpha_{n}(\mathbf{x}_{n}-\mathbf{x}_{n-1}),\\
x_{i,n+1}=w_{i,n}+\varepsilon_{i,n}\lambda_{n}(T_{i,n}(w_{1,n},\cdots,w_{m,n})-w_{i,n}), i=\{1, \cdots,m\},
\end{array}
\right.\eqno{(3.9)}
$$
and set $\mathcal{E}_{n}=\sigma(\varepsilon_{n})$. Moreover, suppose that there exists $\hat{\Omega}\in\mathcal{F}$ such that $P(\hat{\Omega}) = 1$ and the following hold:\\
(i) $\mathbf{D}=\bigcap_{n\in\mathbb{N}}Fix(\mathbf{T}_{n})\neq\emptyset$.\\
(ii) For every $n\in\mathbb{N}$, $\mathcal{E}_{n}$ and $\mathcal{X}_{n}$ are independent.\\
(iii) $P[\varepsilon_{i,0}=1]>0$, $\forall i\in\{1, \cdots,m\}$.\\
(iv) $(\alpha_{n})_{n\geq1}$ is nondecreasing with $\alpha_{1}=0$ and $0\leq\alpha_{n}\leq\alpha<1$, $\forall n\geq1$ and $\lambda, \tau, \delta>0$ are such that
$\delta>\frac{\alpha^{2}(1+\alpha)+\alpha\tau}{1-\alpha^{2}}$ and $0<\max\{\lambda,\frac{b}{\beta_{n}}\}\leq\lambda_{n}\leq\min\{\frac{\delta-\alpha[\alpha(1+\alpha)+\alpha\delta+\tau]}{\delta[1+\alpha(1+\alpha)+\alpha\delta+\tau]}, \frac{(1-b)}{\beta_{n}}\}$, $\forall n\geq1$.\\
(v) $(\forall\omega\in\hat{\Omega})$ $[\beta_{n}^{-1}(\mathbf{T}_{n}\mathbf{x}_{n}(\omega)-\mathbf{x}_{n}(\omega))\rightarrow0\Longrightarrow\mathfrak{Q}(\mathbf{x}_{n})_{n\in\mathbb{N}}\subset\mathbf{D}]$.
Then $(\mathbf{x}_{n})_{n\in\mathbb{N}}$  converges weakly P-a.s. to a $\mathbf{D}$-valued random variable $\mathbf{\hat{x}}$. Furthermore, if\\
(vi)  $(\forall\omega\in\hat{\Omega})$ $[[\sup_{n\in\mathbb{N}}\|\mathbf{x}_{n}(\omega)\|<+\infty ~and ~ \beta_{n}^{-1}(\mathbf{T}_{n}\mathbf{x}_{n}(\omega)-\mathbf{x}_{n}(\omega))\rightarrow0]\Longrightarrow\mathfrak{Y}(\mathbf{x}_{n})_{n\in\mathbb{N}}\neq\emptyset]$,
then $(\mathbf{x}_{n})_{n\in\mathbb{N}}$ converges strongly P-a.s. to $\mathbf{\hat{x}}$.
\end{thm}

\begin{proof}
Set $\mathbf{V}_{n}=(1-\beta_{n}^{-1})I+\beta_{n}^{-1}\mathbf{T}_{n}$ and $\mathbf{V}_{i,n}=(1-\beta_{n}^{-1})I+\beta_{n}^{-1}\mathbf{T}_{i,n}$, $\forall i\in\{1, \cdots,m\}$.
Furthermore, for every $n\in\mathbb{N}$, we set $\vartheta_{n}=\beta_{n}\lambda_{n}$. So, $Fix(\mathbf{V}_{n})=Fix(\mathbf{T}_{n}), \forall n\in\mathbb{N}$ and $\mathbf{V}_{n}$ is
nonexpansive. Hence, for $n\in\mathbb{N}$, from (3.9) we have

$$
 \left\{
\begin{array}{l}
\mathbf{w}_{n}=\mathbf{x}_{n}+\alpha_{n}(\mathbf{x}_{n}-\mathbf{x}_{n-1}),\\
x_{i,n+1}=w_{i,n}+\varepsilon_{i,n}\vartheta_{n}(V_{i,n}(\mathbf{w}_{n})-w_{i,n}), i=\{1, \cdots,m\}.
\end{array}
\right.\eqno{(3.10)}
$$
Therefore, from  Remark 3.3(iii) in [2] and Theorem 3.2, we can obtain the result.

\end{proof}

\begin{rmk}
The binary variable $\varepsilon_{i,n}$ in Theorem 3.2 and Theorem 3.5 signals whether the $i$-th coordinate $T_{i,n}$ of the operator $\mathbf{T}_{n}$ is activated
or not at iteration $n$.
\end{rmk}

\begin{cor}
Let $(\beta_{n})_{n\in\mathbb{N}}$ and  $(\gamma_{n})_{n\in\mathbb{N}}$  be sequences in $(0, 1)$ such that $\sup_{n\in\mathbb{N}}\beta_{n}<1$ and $\sup_{n\in\mathbb{N}}\gamma_{n}<1$. Suppose that $b\in(0, 1)$,  and set $M=\{0,1\}^{m}\backslash \{\mathbf{0}\}$. Let $\mathbf{x}_{0}, \mathbf{x}_{1}$ be  $\mathbf{H}$-valued
random variables which are arbitrarily chosen, and let $(\varepsilon_{n})_{n\in\mathbb{N}}$ be identically distributed $M$-valued random variables.  $\forall n\in\mathbb{N}$, let $\mathbf{V}_{n}: \mathbf{H}\rightarrow \mathbf{H}$ be a $\gamma_{n}$-averaged and $\mathbf{T}_{n}: \mathbf{H}\rightarrow \mathbf{H}: \mathbf{x}\mapsto (T_{i,n}\mathbf{x})_{1\leq i\leq m}$  be a $\beta_{n}$-averaged
operator where,  $\forall i\in\{1,\cdots,m\}$, $T_{i,n} :\mathbf{H }\rightarrow H_{i}$. For $n\geq0$,
$$
 \left\{
\begin{array}{l}
\mathbf{w}_{n}=\mathbf{x}_{n}+\alpha_{n}(\mathbf{x}_{n}-\mathbf{x}_{n-1}),\\
\mathbf{z}_{n}=\mathbf{V}_{n}\mathbf{w}_{n},\\
x_{i,n+1}=w_{i,n}+\varepsilon_{i,n}\lambda_{n}(T_{i,n}(\mathbf{z}_{n})-w_{i,n}), i=\{1, \cdots,m\},
\end{array}
\right.\eqno{(3.11)}
$$
(i) $\mathbf{D}=\bigcap_{n\in\mathbb{N}}Fix(\mathbf{T}_{n}\circ\mathbf{V}_{n})\neq\emptyset$.\\
(ii) For every $n\in\mathbb{N}$, $\mathcal{E}_{n}$ and $\mathcal{X}_{n}$ are independent.\\
(iii) $P[\varepsilon_{i,0}=1]>0$, $\forall i\in\{1, \cdots,m\}$.\\
(iv) $(\alpha_{n})_{n\geq1}$ is nondecreasing with $\alpha_{1}=0$ and $0\leq\alpha_{n}\leq\alpha<1$, $\forall n\geq1$ and $\lambda, \tau, \delta>0$ are such that
$\delta>\frac{\alpha^{2}(1+\alpha)+\alpha\tau}{1-\alpha^{2}}$ and $0<\max\{\lambda,\frac{b}{\beta_{n}}\}\leq\lambda_{n}\leq\min\{\frac{\delta-\alpha[\alpha(1+\alpha)+\alpha\delta+\tau]}{\delta[1+\alpha(1+\alpha)+\alpha\delta+\tau]}, \frac{(1-b)}{\beta_{n}}\}$, $\forall n\geq1$.\\
(v) Assuming that $\mathfrak{Q}(\mathbf{x}_{n})_{n\in\mathbb{N}}\subset\mathbf{D}$ P-a.s. Then $(\mathbf{x}_{n})_{n\in\mathbb{N}}$  converges weakly P-a.s. to a $\mathbf{D}$-valued random variable $\mathbf{\hat{x}}$.

\end{cor}

\begin{proof}
Since  $\mathbf{V}_{n}$ is  a $\gamma_{n}$-averaged operator and $\mathbf{T}_{n}$ is  a $\beta_{n}$-averaged operator, from Lemma 2.4, we can know that $\forall n\in\mathbb{N}$,  $\mathbf{S}_{n}=\mathbf{T}_{n}\circ\mathbf{V}_{n}$ is a
$\eta_{n}=\frac{\beta_{n}+\gamma_{n}-2\beta_{n}\gamma_{n}}{1-\beta_{n}\gamma_{n}}$-averaged operator and $\mathbf{S}_{i,n}=\mathbf{T}_{i,n}\circ\mathbf{V}_{i,n}$, $\forall i\in\{1, \cdots,m\}$.  So, $Fix(\mathbf{V}_{n})=Fix(\mathbf{T}_{n})$, set $\varpi_{n}=\eta_{n}\lambda_{n}$.
Hence, for $n\in\mathbb{N}$, from (3.11) we have

$$
 \left\{
\begin{array}{l}
\mathbf{w}_{n}=\mathbf{x}_{n}+\alpha_{n}(\mathbf{x}_{n}-\mathbf{x}_{n-1}),\\
x_{i,n+1}=w_{i,n}+\varepsilon_{i,n}\varpi_{n}(S_{i,n}(\mathbf{w}_{n})-w_{i,n}), i=\{1, \cdots,m\}.
\end{array}
\right.\eqno{(3.12)}
$$
Therefore, we can obtain the result from  Theorem 3.5.

\end{proof}

\section{A preconditioned stochastic intertial block-coordinate forward-backward algorithm}
First, we introduce some definitions and notations.  We refer the
readers to [21] for more details. Let $\mathbf{\tilde{M}} : \mathbf{H } \rightarrow
\mathbf{H }$ be a set-valued operator. We denote by $ran(\mathbf{\tilde{M}}) := \{v
\in\mathbf{H } : \exists u \in\mathbf{H }, v \in \mathbf{\tilde{M}}u\}$ the range of
$\mathbf{\tilde{M}}$, by $ gra(\mathbf{\tilde{M}}) := \{(u, v) \in \mathbf{H }^{2} : v \in \mathbf{\tilde{M}}u\}$ its
graph, and by $\mathbf{\tilde{M}} ^{-1}$ its inverse; that is, the set-valued
operator with graph $\{(v, u) \in \mathbf{H }^{2} : v \in \mathbf{\tilde{M}}u\}$. We
define $ zer(\mathbf{\tilde{M}}) := \{u \in \mathbf{H } : 0\in \mathbf{\tilde{M}}u\}$. $\mathbf{\tilde{M}}$ is said to
be monotone if $\forall(u, u' ) \in \mathbf{H }^{2},\forall(v, v' )
\in \mathbf{\tilde{M}}u\times \mathbf{\tilde{M}}u'$, $\langle u-u' , v-v' \rangle\geq 0$ and maximally
monotone if there exists no monotone operator $\mathbf{\tilde{M}}'$ such that $
gra(\mathbf{\tilde{M}}) \subset gra(\mathbf{\tilde{M}}') \neq gra(\mathbf{\tilde{M}})$.

The resolvent $(\mathbf{I} + \mathbf{\tilde{M}})^{-1}$ of a maximally monotone operator $\mathbf{\tilde{M}} :
\mathbf{H } \rightarrow \mathbf{H }$ is defined and single-valued on
$\mathbf{H }$ and firmly nonexpansive. The subdifferential $\partial
\mathbf{J}$ of $\mathbf{J}\in \Gamma_{0}(\mathbf{H })$ is maximally monotone.

\begin{thm}
Let $\mathbf{A}: \mathbf{H }\rightarrow2^{\mathbf{H }}$ be a maximally monotone operator and let $\mathbf{B}: \mathbf{H }\rightarrow2^{\mathbf{H }}$ be a cocoercive
operator. Assume that $\mathbf{D }= zer (\mathbf{A}+ \mathbf{B})$ is nonempty. Let $\mathbf{\tilde{L}}$ be a strongly positive self-adjoint
operator in $\mathcal{B}(\mathbf{H})$ such that $\mathbf{\tilde{L}}^{1/2}\mathbf{B}\mathbf{\tilde{L}}^{1/2}$ is $\mu$-cocoercive with $\mu\in(0,+\infty)$. Let $(\theta_{n})_{n\in\mathbb{N}}$ be a sequence
in $\mathbb{R}$ such that $0<\inf_{n\in\mathbb{N}}\theta_{n}\leq\sup_{n\in\mathbb{N}}\theta_{n}<2\mu$ and set $M=\{0,1\}^{m}\backslash \{\mathbf{0}\}$. Let $\mathbf{x}_{0}, \mathbf{x}_{1}$ be  $\mathbf{H}$-valued
random variables which are arbitrarily chosen, and let $(\varepsilon_{n})_{n\in\mathbb{N}}$ be identically distributed $M$-valued random variables. $\forall n\in\mathbb{N}$, set $\mathbf{J}_{\theta_{n}\mathbf{\tilde{L}}\mathbf{A}}: \mathbf{x}\mapsto(T_{i,n}\mathbf{x})_{1\leq i\leq m}$, where $\forall i\in\{1, \cdots,m\}$, $T_{i,n}:\mathbf{H}\rightarrow H_{i}$. For $n\geq0$:
$$
 \left\{
\begin{array}{l}
\mathbf{w}_{n}=\mathbf{x}_{n}+\alpha_{n}(\mathbf{x}_{n}-\mathbf{x}_{n-1}),\\
\mathbf{z}_{n}=\mathbf{\tilde{L}}\mathbf{A}\mathbf{w}_{n},\\
x_{i,n+1}=w_{i,n}+\varepsilon_{i,n}\lambda_{n}(T_{i,n}(\mathbf{w}_{n}-\theta_{n}\mathbf{z}_{n})-w_{i,n}), i=\{1, \cdots,m\}.
\end{array}
\right.\eqno{(4.1)}
$$
(i) For every $n\in\mathbb{N}$, $\mathcal{E}_{n}$ and $\mathcal{X}_{n}$ are independent.\\
(ii) $P[\varepsilon_{i,0}=1]>0$, $\forall i\in\{1, \cdots,m\}$.\\
(iii) $(\alpha_{n})_{n\geq1}$ is nondecreasing with $\alpha_{1}=0$ and $0\leq\alpha_{n}\leq\alpha<1$, $\forall n\geq1$ and $\lambda, \tau, \delta>0$ are such that
$\delta>\frac{\alpha^{2}(1+\alpha)+\alpha\tau}{1-\alpha^{2}}$ and $0<\lambda\leq\lambda_{n}\leq\frac{\delta-\alpha[\alpha(1+\alpha)+\alpha\delta+\tau]}{\delta[1+\alpha(1+\alpha)+\alpha\delta+\tau]}$, $\forall n\geq1$.\\
Then $(\mathbf{x}_{n})_{n\in\mathbb{N}}$  converges weakly P-a.s. to a $\mathbf{D}$-valued random variable $\mathbf{\hat{x}}$.

\end{thm}
\begin{proof}
By assumption, we can know that $zer (\mathbf{\tilde{L}}\mathbf{A}+ \mathbf{\tilde{L}}\mathbf{B})=zer (\mathbf{A}+ \mathbf{B})\neq\emptyset$. Since $\mathbf{\tilde{L}}$ is a strongly positive self-adjoint operator, we can define a particular inner
product $\langle\cdot,\cdot\rangle_{\mathbf{\tilde{L}}^{-1}}$ and norm $\|\cdot\|_{\mathbf{\tilde{L}}^{-1}}=\langle\cdot,\cdot\rangle_{\mathbf{\tilde{L}}^{-1}}^{\frac{1}{2}}$
in $\mathbf{H }$ as
$$\langle\mathbf{x},\mathbf{x}'\rangle_{\mathbf{\tilde{L}}^{-1}}=\langle\mathbf{x},\mathbf{\tilde{L}}^{-1}\mathbf{x}'\rangle, \forall \mathbf{x}, \mathbf{x}'\in \mathbf{H }.\eqno{(4.2)}$$
By endowing $\mathbf{H }$ with this inner product, we obtain the Hilbert space denoted by $\mathbf{H }_{\mathbf{\tilde{L}}^{-1}}$.
In this renormed space, $\mathbf{\tilde{L}}\mathbf{A}$ is maximally monotone. In addition, for every $ \mathbf{x}, \mathbf{x}'\in \mathbf{H }$, from the proof of Proposition 3.1 in [3], we have
\begin{align*}
\|\mathbf{\tilde{L}}\mathbf{B} \mathbf{x}-\mathbf{\tilde{L}}\mathbf{B}\mathbf{x}'\|^{2}_{\mathbf{\tilde{L}}^{-1}}&=\|\mathbf{\tilde{L}}^{\frac{1}{2}}\mathbf{B} \mathbf{x}-\mathbf{\tilde{L}}^{\frac{1}{2}}\mathbf{B}\mathbf{x}'\|^{2}\\
&\leq\mu^{-1}\langle\mathbf{x}-\mathbf{x}',\mathbf{\tilde{L}}\mathbf{B} \mathbf{x}-\mathbf{\tilde{L}}\mathbf{B}\mathbf{x}'\rangle_{\mathbf{\tilde{L}}^{-1}},\tag{4.3}
\end{align*}
which shows that $\mathbf{\tilde{L}}\mathbf{B}$ is $\mu$-cocoercive in $(\mathbf{H }, \|\cdot\|_{\mathbf{\tilde{L}}^{-1}})$. Hence, we can find an  element of $\mathbf{Z }$ by composing operators $\mathbf{J}_{\theta_{n}\mathbf{\tilde{L}}\mathbf{A}}$
and $\mathbf{I}-\theta_{n}\mathbf{\tilde{L}}\mathbf{B}$. Since the first operator is 1/2-averaged and the second one is $\theta_{n}/(2\mu)$-averaged [21, Proposition 4.33]. Noticing that weak convergences in the sense of $\langle\cdot,\cdot\rangle$ and $\langle\cdot,\cdot\rangle_{\mathbf{\tilde{L}}^{-1}}$ are equivalent. So, from the Corollary 3.6, we can obtain the convergence result.
\end{proof}

\section{Applications}

\subsection{Intertial block-coordinate primal-dual algorithms for monotone
inclusion problems}
Let $(E_{j})_{1\leq j\leq p}$ and $(K_{k})_{1\leq k\leq q}$ be separable real
Hilbert spaces, where $p,q$ are positive integers. Furthermore, $\mathbf{E}=E_{1}\oplus\cdots\oplus E_{p}$ and $\mathbf{G}=G_{1}\oplus\cdots\oplus G_{q}$ denote the Hilbert direct
sums of $(E_{j})_{1\leq j\leq p}$ and $(G_{k})_{1\leq k\leq q}$, respectively. We take into account $\mathbf{H }=\mathbf{E}\oplus\mathbf{G}$.
Recently, more and more people pay much attention to the problem involving monotone operators (see e.g. [4-9]), it also play a significant role in our work.
\begin{pro}
Let $A_{j} : E_{j}\rightarrow2^{E_{j}}$ be maximally monotone, and  $C_{j} : E_{j}\rightarrow E_{j}$ be cocoercive, $\forall j\in\{1,\cdots,p\}$. For every  $k\in\{1,\cdots,q\}$, let $B_{k}: G_{k}\rightarrow2^{G_{k}}$ be maximally monotone, let
$\tilde{D}_{k}: G_{k}\rightarrow2^{G_{k}}$ be maximally and strongly monotone, and let $L_{k,j}\in\mathcal{B}(E_{j},G_{k})$. Assuming that
$$\mathbb{L}_{k}=\{j\in\{1,\cdots,p\}|L_{k,j}\neq0\}\neq\emptyset, \forall k\in\{1,\cdots,q\},  \eqno{(5.1)}$$
$$\mathbb{L}_{j}^{\ast}=\{k\in\{1,\cdots,q\}|L_{k,j}\neq0\}\neq\emptyset, \forall j\in\{1,\cdots,p\},  \eqno{(5.2)}$$
and that the set $\mathbf{Z}$ of solutions to the problem:

find $x_{1} \in E_{1},\cdots, x_{p} \in E_{p}$ such that

$$0\in A_{j}x_{j} +C_{j}x_{j}+\sum_{k=1}^{q}L^{\ast}_{k,j}(B_{k}\Box \tilde{D}_{k})(\sum_{j'=1}^{p}L_{k,j'}x_{j'})  \eqno{(5.3)}$$
is nonempty. Furthermore, we  consider the set $\mathbf{Z}^{\ast}$ of solutions to the dual problem:

find $y_{1} \in G_{1},\cdots, y_{q} \in G_{q}$ such that
$$0\in-\sum_{j'=1}^{p}L_{k,j}(A_{j}^{-1} \Box C_{j}^{-1})(-\sum_{k'=1}^{q}L^{\ast}_{k',j}y_{k'})+B_{k}^{-1}y_{k}+ \tilde{D}_{k}^{-1}y_{k}. \eqno{(5.4)}$$
We aim at finding find a pair $(\mathbf{\hat{x}}, \mathbf{\hat{y}})$ of random variables such that $\mathbf{\hat{x}}$ is $\mathbf{Z}$-valued and $\mathbf{\hat{y}}$ is $\mathbf{Z}^{\ast}$-valued.

\end{pro}
By [3, 10-11], we can know that the above problem can be regard as a search for a zero of the sum of two maximally monotone
operators in the product space $\mathbf{H}$.
\begin{lem}
([3]). Let $ \mathbf{A}: \mathbf{E} \rightarrow 2^{\mathbf{E}} : \mathbf{x} \mapsto \times_{j=1}^{p}A_{j}x_{j}$, $ \mathbf{B}: \mathbf{G} \rightarrow 2^{\mathbf{G}} : \mathbf{y} \mapsto \times_{k=1}^{q}B_{k}y_{k}$,
 $ \mathbf{C}: \mathbf{E} \rightarrow 2^{\mathbf{E}} : \mathbf{x} \mapsto (C_{j}x_{j})_{1\leq j\leq p}$, $ \mathbf{\tilde{D}}: \mathbf{G} \rightarrow 2^{\mathbf{G}} : \mathbf{y} \mapsto \times_{k=1}^{q}\tilde{D}_{k}y_{k}$, and
 $ \mathbf{L}: \mathbf{E} \rightarrow\mathbf{G} : \mathbf{x} \mapsto (\sum_{j=1}^{p}L_{k,j}x_{j})_{1\leq k\leq q}$. Now , we consider the operators
 $$
\mathbf{U}: \left(
  \begin{array}{ccccccc}
    \mathbf{x}  \\
    \mathbf{y}  \\

  \end{array}
\right)\mapsto\left(
  \begin{array}{ccccccc}
   \mathbf{A} & \mathbf{L}^{\ast}\\
    -\mathbf{L} &\mathbf{ B}^{-1}\\

  \end{array}
\right)\left(
  \begin{array}{ccccccc}
    \mathbf{x}  \\
    \mathbf{y}  \\

  \end{array}
\right),\eqno{(5.5)}
$$

and
$$
\mathbf{V}: \left(
  \begin{array}{ccccccc}
    \mathbf{x}  \\
    \mathbf{y}  \\

  \end{array}
\right)\mapsto\left(
  \begin{array}{ccccccc}
   \mathbf{C} \\
   \mathbf{ \tilde{D}}^{-1}\\

  \end{array}
\right)\left(
  \begin{array}{ccccccc}
    \mathbf{x}  \\
    \mathbf{y}  \\

  \end{array}
\right).\eqno{(5.6)}
$$

\end{lem}
Then, the following hold:\\
(i) $\mathbf{U}$ is maximally monotone and $\mathbf{V}$ is cocoercive.\\
(ii) $\mathbf{D} = zer (\mathbf{U} + \mathbf{V})$ is nonempty.\\
(iii) A pair $(\mathbf{\hat{x}}, \mathbf{\hat{y}})$ of random variables is a solution to Problem 5.1 if and only if $(\mathbf{\hat{x}}, \mathbf{\hat{y}})$ is $\mathbf{D}$-valued.\\
Now, we will consider an intertial block-coordinate primal-dual algorithms for  Problem 5.1.
\begin{thm}
Let $\mathbf{F}: \mathbf{E }\rightarrow\mathbf{E }:\mapsto(F_{1}x_{1},\cdots,F_{p}x_{p})$  and  $\mathbf{R}: \mathbf{G }\rightarrow\mathbf{G }:\mapsto(R_{1}y_{1},\cdots,R_{q}y_{q})$
where, $\forall j\in\{1,\cdots,p\}$, $F_{j}$ is a strongly positive self-adjoint operator in $\mathcal{B}(E_{j} )$ such that $F_{j}^{\frac{1}{2}}C_{j}F_{j}^{\frac{1}{2}}$ is $\nu_{j}$-cocoercive with $\nu_{j}\in(0,+\infty)$, and
 $\forall k\in\{1,\cdots,q\}$, $R_{k}$ is a strongly positive self-adjoint operator in $\mathcal{B}(G_{k} )$ such that $R_{k}^{\frac{1}{2}}\tilde{D}_{k}R_{k}^{\frac{1}{2}}$ is $\tilde{\tau}_{k}$-cocoercive with $\tilde{\tau}_{k}\in(0,+\infty)$.
 Assume that $(\exists a\in(0,+\infty))$  $2\mu_{a}>1$ where the definition of $\mu_{a}$ is similar with the definition of $\vartheta_{a}$ in Lemma 4.3 [3] with $\nu=\min\{\nu_{1},\cdots,\nu_{p}\}$ and $\tilde{\tau}=\min\{\tilde{\tau}_{1},\cdots,\tilde{\tau}_{q}\}$. Let $\mathbf{x}_{0}, \mathbf{x}_{1}$ be  $\mathbf{E}$-valued
random variables which are arbitrarily chosen, and let $\mathbf{y}_{0}, \mathbf{y}_{1}$ be  $\mathbf{G}$-valued
random variables which are arbitrarily chosen. Let $(\varepsilon_{n})_{n\in\mathbb{N}}$ be identically distributed $M_{p+q}$-valued random variables.  For $n\geq0$:
$$
 \left\{
\begin{array}{l}
for ~j=,\cdots,p,\\
w_{j,n}=x_{j,n}+\alpha_{n}(x_{j,n}-x_{j,n-1}),\\
z_{j,n}=\varepsilon_{j,n}(J_{F_{j}A_{j}}(w_{j,n}-F_{j}(\sum_{k\in\mathbb{L}_{j}^{\ast}}L_{k,j}^{\ast}y_{k,n}+C_{j}w_{j,n}))),\\
x_{j,n+1}=w_{j,n}+\varepsilon_{j,n}\lambda_{n}(z_{j,n}-w_{j,n}),\\
for ~k=,\cdots,q,\\
h_{k,n}=y_{k,n}+\alpha_{n}(y_{k,n}-y_{k,n-1}),\\
s_{k,n}=\varepsilon_{p+k,n}(J_{R_{k}B_{k}^{-1}}(h_{k,n}+R_{k}(\sum_{j\in\mathbb{L}_{k}}L_{k,j}(2z_{j,n}-w_{j,n})-\tilde{D}_{k}^{-1}h_{k,n}))),\\
y_{k,n+1}=h_{k,n}+\varepsilon_{p+k,n}\lambda_{n}(s_{k,n}-h_{k,n}),
\end{array}
\right.\eqno{(5.7)}
$$
and set $\mathcal{E}_{n}=\sigma(\varepsilon_{n})$, $\mathcal{\tilde{X}}_{n}=\sigma(\mathbf{x}_{n'},\mathbf{y}_{n'})_{0\leq n'\leq n}$. Moreover, suppose that the following
hold:\\
(i) For every $n\in\mathbb{N}$, $\mathcal{E}_{n}$ and $\mathcal{\tilde{X}}_{n}$ are independent, and $P[\varepsilon_{p+k,0}=1]>0$, $\forall k\in\{1, \cdots,q\}$.\\
(ii)  For every $j\in\{1, \cdots,p\}$ and $n\in\mathbb{N}$, $\bigcup_{k\in\mathbb{L}_{j}^{\ast}}\{\omega\in\Omega|\varepsilon_{p+k,n}(\omega)=1\}\subset\{\omega\in\Omega|\varepsilon_{j,n}(\omega)=1\}$.\\
(iii) $(\alpha_{n})_{n\geq1}$ is nondecreasing with $\alpha_{1}=0$ and $0\leq\alpha_{n}\leq\alpha<1$, $\forall n\geq1$ and $\lambda, \tau, \delta>0$ are such that
$\delta>\frac{\alpha^{2}(1+\alpha)+\alpha\tau}{1-\alpha^{2}}$ and $0<\lambda\leq\lambda_{n}\leq\frac{\delta-\alpha[\alpha(1+\alpha)+\alpha\delta+\tau]}{\delta[1+\alpha(1+\alpha)+\alpha\delta+\tau]}$, $\forall n\geq1$.\\
Then $(\mathbf{x}_{n})_{n\in\mathbb{N}}$  converges weakly P-a.s. to a $\mathbf{Z}$-valued random variable $\mathbf{\hat{x}}$, and  $(\mathbf{y}_{n})_{n\in\mathbb{N}}$  converges weakly P-a.s. to a $\mathbf{Z}^{\ast}$-valued random variable $\mathbf{\hat{y}}$.

\end{thm}

\begin{proof}
By Lemma 5.1(i)-(ii), we can know that $\mathbf{U}$ is maximally monotone, $\mathbf{V}$ is cocoercive,
and $\mathbf{D}= zer (\mathbf{U}+ \mathbf{V})\neq\emptyset$. On the other hand, $(\exists a\in(0,+\infty))$  $2\mu_{a}>1$ and the definition of $\mu_{a}$ imply that $\|\mathbf{F}^{\frac{1}{2}}\mathbf{L}\mathbf{R}^{\frac{1}{2}}\|<1$. So, with the same idea of Lemma 4.5[3], Algorithm (5.7) can be rewritten under the form of Algorithm
(4.1), where $m = p + q$, $\mathbf{\tilde{L}}$ is defined by (5.8)
$$
\mathbf{\tilde{L}}: \left(
  \begin{array}{ccccccc}
    \mathbf{x}  \\
    \mathbf{y}  \\

  \end{array}
\right)\mapsto\left(
  \begin{array}{ccccccc}
    (\mathbf{F}^{-1}-\mathbf{L}^{\ast}\mathbf{R}\mathbf{L})^{-1} &\mathbf{F}\mathbf{L}^{\ast}(\mathbf{R}^{-1}-\mathbf{L}\mathbf{F}\mathbf{L}^{\ast})^{-1}\\
   (\mathbf{R}^{-1}-\mathbf{L}\mathbf{F}\mathbf{L}^{\ast})^{-1}\mathbf{L}\mathbf{F} &(\mathbf{R}^{-1}-\mathbf{L}\mathbf{F}\mathbf{L}^{\ast})^{-1}\\

  \end{array}
\right)\left(
  \begin{array}{ccccccc}
    \mathbf{x}  \\
    \mathbf{y}  \\

  \end{array}
\right),\eqno{(5.8)}
$$
the more detail about $\mathbf{\tilde{L}}$ can see ([3,Lemma 4.3]), and for every $n\in\mathbb{N}$

$$
\mathbf{x}_{n}=(\mathbf{x}_{n},\mathbf{y}_{n}),\eqno{(5.9)}
$$
$$
\theta_{n}=1,\eqno{(5.10)}
$$
$$
\mathbf{J}_{\mathbf{\tilde{L}}\mathbf{A}}:\mathbf{x}\mapsto(T_{i,n}\mathbf{x})_{1\leq i\leq m},\eqno{(5.11)}
$$
$$
T_{j,n}: \mathbf{H}\rightarrow E_{j}, \forall j\in\{1,\cdots,p\},\eqno{(5.12)}
$$

$$
T_{p+k,n}: \mathbf{H}\rightarrow G_{k}, \forall k\in\{1,\cdots,q\}.\eqno{(5.13)}
$$
From Lemma 4.3(i) in [3] we can know that $\mathbf{\tilde{L}}$ is a strongly positive self-adjoint operator in $\mathcal{B}(\mathbf{H})$.
Therefore, with the same proof of Proposition 4.6 [3], we can know all the assumptions of Theorem 4.1 are satisfied, which allows us to establish the almost sure convergence of $(\mathbf{x}_{n},\mathbf{y}_{n})$
to a $\mathbf{D}$-valued
random variable. Finally, Lemma 5.1(iii) ensures that the limit is an $\mathbf{Z}\times\mathbf{Z}^{\ast}$-valued
random variable.

\end{proof}

\subsection{Intertial block-coordinate primal-dual algorithms for convex optimization
problems}
In this section, we will introduce an intertial block-coordinate primal-dual algorithms for solving a wide range of structured convex optimization
problems. The results obtained in the previous section. In particular, we will pay our attention to the following
optimization problems. We still use the notation of the previous section and present some new notations. We denote by $\Gamma_{0}(H)$ the class of lower semicontinuous convex functions
$f:H\rightarrow(-\infty,+\infty)$ such that $f\neq+\infty$. The Moreau subdifferential of $f\in\Gamma_{0}(H)$  is the maximally monotone operator
$$
\partial f: H\rightarrow2^{H}:x\mapsto\{u\in H|\langle y-x,u\rangle+f(x)\leq f(y)\}  .\eqno{(5.14)}
$$

\begin{defn}
 Let $f$ be  a real-valued convex function on
$H$, the operator prox$_{f}$ is defined by
\begin{align*}
prox_{f}&:H\rightarrow H\\
& x\mapsto \arg \min_{y\in
\mathcal{X}}f(y)+\frac{1}{2}\|x-y\|_{2}^{2},
\end{align*}
called the proximity operator of $f$.
\end{defn}
For more details about convex analysis and monotone operator theory, see[21].
\begin{pro}
$\forall j\in\{1,\cdots,p\}$, let $f_{j}, h_{j}\in\Gamma_{0}(E_{j})$, and $h_{j}$ be Lipschitz-differentiable. $\forall k\in\{1,\cdots,q\}$, let $g_{k}, l_{k}\in\Gamma_{0}(G_{k})$, and $l_{k}$ be strongly convex. Let
$L_{k,j}\in\mathcal{B}(E_{j},G_{k})$. Assume that (5.1) and (5.2) hold, and that there exists $(\bar{x}_{1},,\cdots,\bar{x}_{p})\in E_{1}\oplus\cdots\oplus E_{p}$ such that\\
$$
0\in\partial f_{j}(\bar{x}_{j})+\nabla h_{j}(\bar{x}_{j})+\sum_{k=1}^{q}L_{k,j}^{\ast}(\partial g_{k}\Box\partial l_{k})(\sum_{j'=1}^{p}L_{k,j'}\bar{x}_{j'}),\,\,\forall j\in\{1,\cdots,p\}.\eqno{(5.15)}
$$

Let $\mathbf{\tilde{Z}}$ be the set of solutions to the problem
$$
 \min_{x_{1}\in E_{1},\cdots,x_{p}\in E_{p}}\sum_{j=1}^{p}(f_{j}(x_{j})+ h_{j}(x_{j}))+\sum_{k=1}^{q}( g_{k}\Box l_{k})(\sum_{j=1}^{p}L_{k,j}x_{j}),\eqno{(5.16)}
$$
and let $\mathbf{\tilde{Z}}^{\ast}$ be the set of solutions to the dual problem
$$
 \min_{y_{1}\in G_{1},\cdots,y_{q}\in G_{q}}\sum_{j=1}^{p}(f_{j}^{\ast}\Box h_{j}^{\ast})(-\sum_{k=1}^{q}L_{k,j}^{\ast}y_{k})+\sum_{k=1}^{q}( g_{k}^{\ast}(y_{k})+ l_{k}^{\ast}(y_{k})).\eqno{(5.17)}
$$
We aim at finding find a pair $(\mathbf{\hat{x}}, \mathbf{\hat{y}})$ of random variables such that $\mathbf{\hat{x}}$ is $\mathbf{\tilde{Z}}$-valued and $\mathbf{\hat{y}}$ is $\mathbf{\tilde{Z}}^{\ast}$-valued.

\end{pro}
In order to satisfy the condition in Problem 5.2, we need the following assumptions:
\begin{prop}
([10, Proposition 5.3]). Consider the setting of Problem 5.2. Suppose that (5.16) has
a solution. Then, the existence of $(\bar{x}_{1},,\cdots,\bar{x}_{p})\in E_{1}\oplus\cdots\oplus E_{p}$ satisfying (5.15) is guaranteed in each
of the following cases:\\
(i) $\forall j\in\{1,\cdots,p\}$,  $f_{j}$ is real-valued and $\forall k\in\{1,\cdots,q\}$, $(x_{j})_{1\leq j\leq p}\mapsto L_{k,j}x_{j}$
is surjective.\\
(ii) $\forall k\in\{1,\cdots,q\}$, $g_{k}$ or $l_{k}$ is real-valued.\\
(iii) $(E_{j})_{1\leq j\leq p}$ and $(G_{k})_{1\leq k\leq q}$ are finite-dimensional, and $\exists x_{j}\in$ ri dom $f_{j}$ such
that $L_{k,j}x_{j}\in$ ri dom $g_{k}$+ri dom $l_{k}$.

\end{prop}
The following result can be deduced from Theorem 5.1:
\begin{thm}
Let $\mathbf{F}$  and  $\mathbf{R}$ be defined as in Theorem 5.1.
 $\forall j\in\{1,\cdots,p\}$, let  $\nu_{j}^{-1}\in(0,+\infty)$ be a Lipschitz constant of the gradient of $h_{j}\circ F_{j}^{\frac{1}{2}}$, and
 $\forall k\in\{1,\cdots,q\}$, let $\tilde{\tau}_{k}^{-1}\in(0,+\infty)$ be a Lipschitz constant of the gradient of $l_{k}^{\ast}\circ R_{k}^{\frac{1}{2}}$.
 Assume that $(\exists a\in(0,+\infty))$  $2\mu_{a}>1$ where the definition of $\mu_{a}$ is similar with the definition of $\vartheta_{a}$ in Lemma 4.3 [3] with $\nu=\min\{\nu_{1},\cdots,\nu_{p}\}$ and $\tilde{\tau}=\min\{\tilde{\tau}_{1},\cdots,\tilde{\tau}_{q}\}$. Let $\mathbf{x}_{0}, \mathbf{x}_{1}$ be  $\mathbf{E}$-valued
random variables which are arbitrarily chosen, and let $\mathbf{y}_{0}, \mathbf{y}_{1}$ be  $\mathbf{G}$-valued
random variables which are arbitrarily chosen. Let $(\varepsilon_{n})_{n\in\mathbb{N}}$ be identically distributed $M_{p+q}$-valued random variables.  For $n\geq0$:
$$
 \left\{
\begin{array}{l}
for ~j=,\cdots,p,\\
w_{j,n}=x_{j,n}+\alpha_{n}(x_{j,n}-x_{j,n-1}),\\
z_{j,n}=\varepsilon_{j,n}(prox_{f_{j}}^{F_{j}^{-1}}(w_{j,n}-F_{j}(\sum_{k\in\mathbb{L}_{j}^{\ast}}L_{k,j}^{\ast}y_{k,n}+\nabla h_{j}(w_{j,n})))),\\
x_{j,n+1}=w_{j,n}+\varepsilon_{j,n}\lambda_{n}(z_{j,n}-w_{j,n}),\\
for ~k=,\cdots,q,\\
\tilde{h}_{k,n}=y_{k,n}+\alpha_{n}(y_{k,n}-y_{k,n-1}),\\
s_{k,n}=\varepsilon_{p+k,n}(prox_{g_{k}^{\ast}}^{B_{k}^{-1}}(\tilde{h}_{k,n}+R_{k}(\sum_{j\in\mathbb{L}_{k}}L_{k,j}(2z_{j,n}-w_{j,n})-\nabla l_{k}^{\ast}(\tilde{h}_{k,n})))),\\
y_{k,n+1}=\tilde{h}_{k,n}+\varepsilon_{p+k,n}\lambda_{n}(s_{k,n}-\tilde{h}_{k,n}),
\end{array}
\right.\eqno{(5.18)}
$$
and set $\mathcal{E}_{n}=\sigma(\varepsilon_{n})$, $\mathcal{\tilde{X}}_{n}=\sigma(\mathbf{x}_{n'},\mathbf{y}_{n'})_{0\leq n'\leq n}$. Moreover, suppose that Conditions (i)-(iii) in Theorem 5.1 hold.\\
Then $(\mathbf{x}_{n})_{n\in\mathbb{N}}$  converges weakly P-a.s. to a $\mathbf{\tilde{Z}}$-valued random variable $\mathbf{\hat{x}}$, and  $(\mathbf{y}_{n})_{n\in\mathbb{N}}$  converges weakly P-a.s. to a $\mathbf{\tilde{Z}}^{\ast}$-valued random variable $\mathbf{\hat{y}}$.
\end{thm}
\begin{proof}
$\forall j\in\{1,\cdots,p\}$, we set $A_{j}=\partial f_{j}$, $C_{j}=\nabla h_{j}$, and $\forall k\in\{1,\cdots,q\}$, $B_{k}=\partial g_{k}$, $D_{k}^{-1}=\nabla l_{k}^{\ast}$. Observing that $\forall j\in\{1,\cdots,p\}$ and
 $\forall k\in\{1,\cdots,q\}$, $J_{F_{j}A_{j}}=prox_{f_{j}}^{F_{j}^{-1}}$, $J_{R_{k}B^{-1}_{k}}=prox_{g_{k}^{\ast}}^{B_{k}^{-1}}$. In addition, the Lipschitz-differentiability assumptions made
on $h_{j}$ and $l_{k}^{\ast}$ are equivalent to the fact that $F_{j}^{\frac{1}{2}}C_{j}F_{j}^{\frac{1}{2}}$ is $\nu_{j}$-cocoercive and $R_{k}^{\frac{1}{2}}\tilde{D}_{k}^{-1} R_{k}^{\frac{1}{2}}$ is $\tilde{\tau}_{k}$-cocoercive.[21, Corollaries 16.42, 18.16]. The convergence result follows from [1, Proposition 5.3].

\end{proof}
In Problem 5.2, if $\forall j\in\{1,\cdots,p\}$, $f_{j}=0$, we can obtain the following  Corollary.
\begin{cor}
Let $\mathbf{F}$  and  $\mathbf{R}$ be defined as in Theorem 5.1.
 Let $\nu$ and $\tilde{\tau}$ be defined as in Theorem 5.4. Suppose that Condition $\min\{\nu,\tilde{\tau}(1-\|\mathbf{R}^{\frac{1}{2}}\mathbf{L}\mathbf{F}^{\frac{1}{2}}\|^{2})\}>\frac{1}{2}$ holds. Let $\mathbf{x}_{0}, \mathbf{x}_{1}$ be  $\mathbf{E}$-valued
random variables which are arbitrarily chosen, and let $\mathbf{y}_{0}, \mathbf{y}_{1}$ be  $\mathbf{G}$-valued
random variables which are arbitrarily chosen. Let $(\varepsilon_{n})_{n\in\mathbb{N}}$ be identically distributed $M_{p+q}$-valued random variables.  For $n\geq0$:
$$
 \left\{
\begin{array}{l}
for ~j=,\cdots,p,\\
\xi_{j,n}=\max\{\varepsilon_{p+k,n}|k\in \mathbb{L}_{j}^{\ast}\},\\
w_{j,n}=x_{j,n}+\alpha_{n}(x_{j,n}-x_{j,n-1}),\\
z_{j,n}=\xi_{j,n}(w_{j,n}-F_{j}\nabla h_{j}(w_{j,n})),\\
for ~k=,\cdots,q,\\
\tilde{h}_{k,n}=y_{k,n}+\alpha_{n}(y_{k,n}-y_{k,n-1}),\\
s_{k,n}=\varepsilon_{p+k,n}(prox_{g_{k}^{\ast}}^{B_{k}^{-1}}(\tilde{h}_{k,n}+R_{k}(\sum_{j\in\mathbb{L}_{k}}L_{k,j}(z_{j,n}-F_{j}\sum_{k'\in\mathbb{L}_{j}^{\ast}}L_{k',j}^{\ast}y_{k',n})-\nabla l_{k}^{\ast}(\tilde{h}_{k,n})))),\\
y_{k,n+1}=\tilde{h}_{k,n}+\varepsilon_{p+k,n}\lambda_{n}(s_{k,n}-\tilde{h}_{k,n}),\\
for ~j=,\cdots,p,\\
x_{j,n+1}=w_{j,n}+\varepsilon_{j,n}\lambda_{n}(z_{j,n}-F_{j}\sum_{k\in\mathbb{L}_{j}^{\ast}}L_{k,j}^{\ast}s_{k,n}-w_{j,n}),
\end{array}
\right.\eqno{(5.19)}
$$
and set $\mathcal{E}_{n}=\sigma(\varepsilon_{n})$, $\mathcal{\tilde{X}}_{n}=\sigma(\mathbf{x}_{n'},\mathbf{y}_{n'})_{0\leq n'\leq n}$. Moreover, suppose that Condition (i) and (iii) in Theorem 5.1 is satisfied
and the following
hold:\\
 For every $k\in\{1, \cdots,q\}$ and $n\in\mathbb{N}$, $\bigcup_{j\in\mathbb{L}_{k}}\{\omega\in\Omega|\varepsilon_{j,n}(\omega)=1\}\subset\{\omega\in\Omega|\varepsilon_{p+k,n}(\omega)=1\}$.\\
Then $(\mathbf{x}_{n})_{n\in\mathbb{N}}$  converges weakly P-a.s. to a $\mathbf{\tilde{Z}}$-valued random variable $\mathbf{\hat{x}}$, and  $(\mathbf{y}_{n})_{n\in\mathbb{N}}$  converges weakly P-a.s. to a $\mathbf{\tilde{Z}}^{\ast}$-valued random variable $\mathbf{\hat{y}}$.

\end{cor}

\begin{rmk}
Our results improve and extend the results of other people in the following aspects.\\
(i) If $\alpha_{n}=0, \forall n\in \mathbb{N}$ in Theorem 5.1,  Theorem 5.4 and Corollary 5.5,  we can obtain the Proposition 4.6, Proposition 5.3 and Proposition 5.4 of Jean-Christophe and Audrey [3] in the absence of errors.\\
(ii) If $\forall n\in \mathbb{N}$, $\alpha_{n}=0$ and $p=1$, Algorithm (5.18) extends the deterministic approaches in [11-15] by introducing some random sweeping of the coordinates in the absence of errors. Similarly, when  $\alpha_{n}=0$, $p=q=1$, $l_{1}=\iota_{\{0\}}$, $F_{1}=\bar{\tau} I$, with $\bar{\tau}\in(0,+\infty)$, $R_{1}=\bar{\rho} I$  with $\bar{\rho}\in(0,+\infty)$,  $E_{1}$ and $G_{1}$ are finite dimensional spaces and  $\lambda_{n}\equiv1, n\in\mathbb{N}$,  Algorithm (5.19) extends the algorithms
in [16-17] which were developed in a deterministic setting.\\
(iii)If $\forall k\in\{1,\cdots,q\}, B_{k}=\tilde{D}_{k}=$, $p=1$ and $\lambda_{n}\equiv1, n\in\mathbb{N}$,  Algorithm (5.7) extends the deterministic approaches in [18] by introducing some random sweeping of the coordinates in the absence of errors.\\
(iv) Theorem 3.2 extends the corresponding results Theorem 5 of Radu Ioan, Ern\"{o} Robert and Christopher [1] from a
nonexpansive mapping to a quasinonexpansive
mapping.\\
(v) In Theorem 2.1, Corollary 2.3, Theorem 3.2 and Theorem 3.5, if $\alpha_{n}=0, \forall n\in \mathbb{N}$,  we can obtain Theorem 2.5, Corollary 2.7, Theorem 3.2 and Corollary 3.8 of Patrick,  Combettes and Jean-Christophe [2] in the absence of errors.
\end{rmk}

\bigskip
\bigskip

\noindent \textbf{Acknowledgements}

This work was supported by the National Natural Science Foundation
of China (11131006, 41390450, 91330204, 11401293), the National
Basic Research Program of China (2013CB 329404), the Natural Science
Foundations of Jiangxi Province (CA20110\\
7114, 20114BAB 201004).

\end{document}